\documentclass[reqno,12pt]{amsart}
\usepackage{amsmath,amssymb,amsfonts,amsthm, amscd,indentfirst}
\usepackage{amsmath,latexsym,amssymb,amsmath,
	amscd,amsthm,amsxtra}
\usepackage{color}
\usepackage{pdfpages}
\usepackage{graphics}
\usepackage{enumitem}

\usepackage[usenames,dvipsnames]{xcolor}

\usepackage{varioref}
\usepackage{hyperref}
\usepackage[noabbrev,capitalise,nameinlink]{cleveref}

\hypersetup{colorlinks={true},linkcolor={blue},citecolor=blue}
\usepackage[english]{babel}
\usepackage{csquotes}

\usepackage[backend=biber,
style=numeric,datamodel=ext-eprint,maxbibnames=99]{biblatex}
\DeclareSourcemap{
	\maps[datatype=bibtex]{
		\map{
			\step[fieldsource=pmid, fieldtarget=pubmed]
		}
	}
}

\usepackage{graphicx}
\usepackage{caption}

\usepackage{epsfig,here}
\usepackage{subfigure,here}

\newtheorem{theorem}{Theorem}[section]
\newtheorem{lemma}[theorem]{Lemma}
\newtheorem{proposition}[theorem]{Proposition}
\newtheorem{corollary}[theorem]{Corollary}
\newtheorem{remark}[theorem]{Remark}

\def\rr{\mathbb{R}}

\def\bb{\mathbb{B}}

\def\O{\Omega}
\def\p{\partial}

\def\a{\alpha}

\def\d{\delta}

\def\p{\partial}

\def\S{{\Sigma}}
\def\<{\langle}
\def\>{\rangle}
\def\div{{\rm div}}
\def\n{\nabla}

\def\R{{\mathbb R}}

\def\p{\partial}

\newcommand{\mbR}{\mathbb{R}}

\numberwithin{equation} {section}

\begin{document}
	
	\title[Heintze-Karcher inequality]{Capillary hypersurfaces, Heintze-Karcher's inequality and Zermelo's navigation}

		\author{Guofang Wang} \thanks{${}^*$GW is the corresponding author}
	\address{Mathematisches Institut\\
		Universit\"at Freiburg\\
	Ernst-Zermelo-Str.1,
		79104, Freiburg, Germany}
	\email{guofang.wang@math.uni-freiburg.de}
	
	\author{Chao Xia}
	\address{School of Mathematical Sciences\\
		Xiamen University\\
		361005, Xiamen, P.R. China}
	\email{chaoxia@xmu.edu.cn}
\thanks{CX is  supported by the NSFC (Grant No. 12271449, 12126102)}

	\begin{abstract}
	
	In this paper, we establish a Heintze-Karcher-type inequality for capillary hypersurfaces in a unit ball. To achieve this, we introduce a special Finsler metric given by Zermelo's navigation and study the geodesic normal flow with respect to this Finsler metric.  Our results indicate that the relationship between capillary hypersufaces  and hypersurfaces with free boundary is similar to the one between Finsler geometry and Riemannian geometry.
		
		\


	\end{abstract}
	
	\maketitle
	
	\medskip

\section{introduction}

The celebrated Alexandrov Theorem, which states  that
\textit{a closed embedded CMC hypersurface in the Euclidean space is a round sphere,}
plays an important role in differential geometry. Among several interesting proofs, one  is based on the Heintze-Karcher inequality
 \begin{equation*}
 \label{eq1}
 \int_\S \frac 1 H \ge \frac {n+1}n |\Omega|,
 \end{equation*}
 for any strictly mean-convex hypersurface $\S\subset\mathbb{R}^{n+1}$, where $\Omega \subset\R^{n+1}$ is the enclosed domain of $\S$ and $H$ is the mean curvature, with equality holding only on round spheres. This sharp inequality was essentially contained in the seminal paper of Heintze-Karcher \cite{HK78} and was given in this form by Ros \cite{Ros87}. Since then, the Heintze-Karcher inequality becomes one of fundamental geometric inequalities among the isoperimetric inequality and the Alexandrov-Fenchel inequalities.
 
 There are two main approaches to establish Heintze-Karcher inequalities.
  One uses a Reilly-type formula, which was first used by Ros \cite{Ros87} and then by Qiu-Xia \cite{QX15} and Li-Xia \cite{LX19} for various Heintze-Karcher-type inequalities in Riemannian manifolds with certain condition.
See also \cite{MP19}, \cite{FP22},  and also \cite{HMR18, De14}, where a Reilly-type formula for spinors  was  used.
 Another method, which we will follow closely in this paper  and extend, was introduced in the original paper of Heintze-Karcher \cite{HK78} by considering a (singular) foliation of suitable parallel hypersurfaces. 
 In warped product manifolds, Brendle \cite{Br13} generalized this idea  by considering a conformal flow with respect to the warped product factor $f$, which coincides with the geodesic normal flow with respect to a conformal metric determined by $f$. 
 
 In this paper, we will extend this idea further to establish a Heintze-Karcher inequality for capillary hypersurfaces in the unit Euclidean ball $\mathbb{B}^{n+1}$, which remains open since our work on free boundary hypersurfaces in \cite{WX19}.
It will be interesting to see that in order to establish such an inequality it is natural to  work with Finsler metrics. In fact, instead of using the conformal flow by Brendle, we will use geometric flows with respect to a Finsler metric of Randers-type which is constructed suitably for the study of capillary hypersurfaces in the unit ball. This Randers metric solves Zermelo's navigation problem generated by the hyperbolic metric and a constant vector field. 

In order to state our results more precisely, let us begin to introduce capillary hypersurfaces. For a hypersurface with boundary on a given support hypersurface, one can define a contact angle at any boundary point, $\theta(x)$.
Such a hypersurface is called a $\theta_0$-capillary hypersurface if $\theta(x)=\theta_0$ $\forall x\in \p\Sigma$ for some $\theta_0\in (0,\pi)$ . 
In the special case $\theta_0=\frac{\pi}{2}$, it is called a hypersurface with free boundary or a free boundary hypersurface. 
We will work in the Euclidean half ball $\overline{\mathbb{B}^{n+1}}\cap \rr^{n+1}_+$
, where $\mathbb{B}^{n+1}=\{|x|<1\}$ and $\rr^{n+1}_+=\{x_{n+1}>0\}$. The following Heintze-Karcher-type inequality has been proved in  \cite{WX19}.

 \
 
 \noindent{\bf Theorem A} (\cite{WX19})
{\it Let $\S\subset \overline{\mathbb{B}^{n+1}}\cap \rr^{n+1}_+$ be a compact hypersurface with  boundary $\p\S\subset \p \mathbb{B}^{n+1}$.
such that $\p\S$ intersects $\p \mathbb{B}^{n+1}$ transversally with  contact angle  $\theta(x)\le \frac{\pi}{2}$ for any $x\in \p\S$. 
Assume that $\S$ is strictly mean convex. Then
\begin{eqnarray} \label{eq_a1.1}
\int_\S \frac{x_{n+1}}{H}\ge \frac{n+1}{n}\int_\O x_{n+1}.
\end{eqnarray} Moreover, equality holds if and only if
$\Sigma$ is a free boundary spherical cap.
}

\ 

Theorem A has been proved by a weighted Reilly-type formula by Qiu-Xia \cite{QX15} applying on the solution to a mixed boundary value problem.
By the same approach, the following form has been proved by Jia-Xia-Zhang \cite{JXZ22} for capillary hypersurfaces:
	\begin{align}
 \label{EQ-HK-halfball}
			\int_\S\frac{x_{n+1}}{H} \ge\frac{n+1}{n} \int_\O x_{n+1} -\cos\theta_0\frac{\left(\int_\S \<\nu, E_{n+1}\>\right)^2}{\int_\S H\<\nu, E_{n+1}\> },
		\end{align}
where $\nu$ is the unit outward normal to $\S$ and  $E_{n+1}$ is the $(n+1)$-$th$ unit coordinate vector. However, in view of the Minkowski-type formula
\begin{equation}\label{mink-halfspace3}
		\int_\S n(x_{n+1}+\cos\theta_0\<\nu, E_{n+1}\>) -H\<X_{n+1},\nu\>=0,\end{equation}
   that has been established by the authors \cite{WX19}, a more natural form of Heintze-Karcher inequality, \eqref{eq_a1} below, is expected. Here $X_{n+1}=x_{n+1}x-\frac12(|x|^2+1)E_{n+1}$ is a conformal Killing field,
   which played a crucial role in the classification of stable CMC capillary hypersurfaces in \cite{WX19}.
   Moreover it holds that $$\int_\S \<X_{n+1},\nu\>=(n+1)\int_\O x_{n+1}.$$ 
   Our main result in this paper is the following.

  \begin{theorem} \label{thm1.2}
Let $\theta_0\in (0, \pi)$ and $\S\subset \overline{\mathbb{B}^{n+1}}\cap \rr^{n+1}_+$ be a compact hypersurface with boundary $\p\S\subset \p \mathbb{B}^{n+1}$ such that \begin{eqnarray}\label{assump}
\S\subset \{x_{n+1}>|\cos\theta_0|\}
\end{eqnarray}
and $\p\S$ intersects $\p \mathbb{B}^{n+1}$ transversally with  contact angle  $\theta(x)\le \theta_0$ for each $x\in \p\S$. Assume that $\S$ is strictly mean convex. Let $\O$ be the domain enclosed by $\S$ and $\mathbb{S}^n_+$.  Then
\begin{eqnarray}\label{eq_a1}
\int_\S \frac{x_{n+1}+\cos\theta_0\<\nu, E_{n+1}\>}{H}\ge \frac{n+1}{n}\int_\O x_{n+1}.
\end{eqnarray}Moreover, equality holds if and only if
$\Sigma$ is a $\theta_0$-capillary spherical cap.
\end{theorem}

The proof of Theorem \ref{thm1.2} shares a similar idea of Heintze-Karcher \cite{HK78} and Brendle \cite{Br13} by considering a foliation of suitable parallel hypersurfaces, or geodesic normal flow. When $\theta_0=\frac{\pi}{2}$, i.e. the free boundary case,
our initial observation is that instead of the Euclidean metric $\d$ one needs to use the Poincar\'e hyperbolic metric $g_h=x_{n+1}^{-2}\delta$ to define a geodesic flow on $\overline{\mathbb{B}^{n+1}}\cap \rr^{n+1}_+$.
The choice of the metric $g_h$ matches the weight in \eqref{eq_a1.1}.
Under this flow, the Heintze-Karcher deficit $\int_\S \frac{x_{n+1}}{H}- \frac{n+1}{n}\int_\O x_{n+1}$ has a good monotone property. Using this way we provide a different proof of Theorem A in Section 2.

When $\theta_0\neq\frac{\pi}{2}$, i.e., the general capillary case, instead of a Riemannian metric we need to use a Finsler metric to define a geodesic flow. Then the problem has a very close relationship to  Zermelo's navigation problem, which we shall make a brief introduction below. As in the free boundary case we endow the half space $\R^{n+1}_+$ the Poincar\'e hyperbolic metric $g_h$. Together with
the constant vector field ${\bf v_0}=\cos\theta_0 E_{n+1}$, $(g_h, {\bf  v_0})$ gives a navigation data, 
which generates a Randers metric 
$F$ given explicitly by \eqref{finsler}. 
The advantage of the use of this metric is that the  
capillary boundary condition in the Euclidean ball can be transferred to a free boundary condition with respect to this Finsler metric.
The correct geodesic normal flow we shall consider in this case is exactly the one with respect to this Finsler metric, and under this flow we see that the Heintze-Karcher deficit
$\int_\S \frac{x_{n+1}+\cos\theta_0\<\nu, E_{n+1}\>}{H}- \frac{n+1}{n}\int_\O x_{n+1}$ has also a good monotone property. However in this case we meet another difficulty,
the proof of  the surjectivity of the geodesic normal flow in the enclosed domain $\O$. The disadvantage of the Finsler metric is that we do not know how its geodesics behave. Nevertheless, we are able to overcome this difficulty and prove the surjectivity  by using the convexity of the ball under this metric. Finally we remark that the assumption \eqref{assump} is to ensure the well-definedness for the Finsler metric \eqref{finsler} generated by navigation data $(g_h, {\bf  v_0})$. If $\S$ is convex, we can make a rotation of $\S$ such that \eqref{assump} holds true, see \cite[Proposition 2.16]{WeX22}.

It is interesting to see that Finsler geometry is used to prove a result in Riemannian geometry.
{The approach to Theorem \ref{thm1.2} is widely applicable. In particular, one can reprove the Heintze-Karcher inequality for capillary hypersurfaces in the half-space that was established by the authors together with Jia and Zhang in \cite{JWXZ22}. We will make a brief comment in Section \ref{4.3}. We hope the approach could be applied to other capillary boundary problems. }
  
It would be an interesting question to ask whether one can use a Reilly type formula to prove 
Theorem \ref{thm1.2}. With a usual Reilly type formula one can only prove \eqref{EQ-HK-halfball} mentioned above, which is weaker than \eqref{eq_a1}, at least in the convex case. We expect that one needs to obtain a Reilly formula by  using a Finsler Laplacian operator with respect to the Randers metric we use in this paper,
instead the ordinary Laplacian  operator w.r.t a Riemannian metric.


Let us end the introduction by briefly introducing Zermelo's navigation problem. In \cite{Zer31} Zermelo asked: Suppose a ship sails on a calm sea with a mild breeze. How the captain steers the ship to reach a destination in the shortest time. He answered the question if the sea is  $\R^2$, with the Euclidean metric. When the sea is a general Riemannian manifold $(M, g)$ with a time-independent wind $v_0$, Shen \cite{Shen03} gave an answer by using the Randers metric
\[ F (x,y)=\sqrt{\a(x,y)}+\beta(x,y)= \lambda^{-1}\sqrt{\lambda g(y,y)+g(v_0, y)^2} -\lambda^{-1} g(v_0, y), \quad (x, y)\in TM\]
where $\lambda =1-g(v_0,v_0)$.
In fact, he showed that the paths minimizing travel time are exactly the geodesics of this Randers metric. $F$ is called the Finsler metric generated by the navigation data $(g, v_0)$. We refer to \cite{BRS04} for more information and references therein. 
Zermelo's navigation problem has  many applications not only in mathematics, but also in natural sciences.  Here we just mention a few,   \cite{GHWW09} in stationary spacetimes, 
\cite{GW10} for  a model for traffic noise with  magnetic flow,
 \cite{B3LG19} for 2D turbulent flows,
 \cite{D22} in wildfire spread modelling, and  \cite{RS15, Brody15}
in the quantum Zermelo problem \cite{RS15, Brody15}.
Our result is a nice geometric application of it.

The rest of the paper is organized as follows. 
In Section 2, we consider the special case $\theta_0=\frac{\pi}{2}$ and prove Theorem A by using geodesic normal flow with respect to hyperbolic metric. In Section 3, we review fundamental facts about Zermelo's navigation and Finsler metric of Randers-type. In Section 4, we make a detailed description for the relationship between capillary hypersurfaces and Zermelo's navigation and then prove Theorem \ref{thm1.2}.

\

\section{Free boundary hypersurfaces in \texorpdfstring{$\mathbb{B}^{n+1}_+$}{Lg}}


In this section, as a warm-up, we consider the case $\theta_0=\pi/2$. As mentioned  in the introduction, Theorem A was proved in \cite{WX19} by using a Reilly-type formula due to Qiu-Xia \cite{QX15}. Here we provide another proof  by adapting Brendle's idea to hypersurfaces with boundary.

Let $\theta_0\in (0,\pi)$. Let $\mathbb{B}^{n+1}\cap \rr^{n+1}_+\subset\rr^{n+1}_+$ be the open half ball. {We denote $\bar N$ be the outward Euclidean unit normal to $\p \mathbb{B}^{n+1}$, that is, $\bar N(p)=p$ for $p\in \p \mathbb{B}^{n+1}$.  }

Recall $g:=g_h=x_{n+1}^{-2}\delta$ be the hyperbolic metric on $\rr^{n+1}_+$.
For each $p\in \S$,  let $\nu_\delta$ and $\nu_g$ be the outward unit normal to $\S$ with respect to $\delta$ and $g$ respectively.  Let $\exp_p^g$  the exponential map at $p$ with respect to $g$, and $\Phi$ be the normal exponential map from $\S$ given by\begin{equation}
\begin{array}{rl}\vspace{2mm}\Phi: &\S\times [0,\infty)\to \rr^{n+1}_+,\\
&\Phi(p, t)=\exp_p^g(-t\nu_g(p)).\end{array} \label{geo1}
\end{equation}
Let $d_g(\S, q)$ be the distance function of $q\in  \rr^{n+1}_+$ from $\S$ and $c: \S\to [0,\infty)$ be the cut function of $\S$, that is, $$c(p)=\sup\{t\in [0, \infty): d_g(\S, \Phi(p, t))=t\}.$$ Denote $T=\max\limits_{p\in \S} c(p)$.
We define
\begin{eqnarray*}
&A=\{(p, t)\in \S\times [0, T): d_g(\S, \Phi(p, t))=t\},\\
&A^*=\{(p, t)\in \S\times [0, T): (p, t+\delta)\in A\hbox{ for some }\delta>0\}.
\end{eqnarray*}
It is known that $A^*$ is an open set of $\S\times [0, T)$ and the restriction $\Phi|_{A^*}$ is a diffeomorphism.
$\Phi(A)$ is a closed set and $\Phi(A\setminus A^*)$ is the $g$-cut locus ${\rm Cut}_g(\S)$, which is of ${\mathcal{H} ^{n+1}}$-measure zero, see for example \cite{IT01}.

For each $t\in [0, T)$, define $$\S_t^*=\Phi(A^*\cap (\S\times\{t\})),$$
which is in fact a parallel hypersurface of $\Sigma$.
Then $\S_t^*$ is a smooth hypersurface contained in $\{d_g(\S, \cdot)=t\}$. Let $\nu_\delta(\cdot, t)$ and $\nu_g(\cdot, t)$ be the outward $\delta$-unit normal and $g$-unit normal  to $\S_t^*$ respectively.  
It is clear that $\S_t^*$ solves the flow equation 
\begin{equation}\label{conf_flow}\p_t \Phi(p, t) =-\nu_g(p, t)=-x_{n+1}\nu_\delta(p, t), \end{equation}
which is a unit-speed normal flow w.r.t. the  metric $g$, or a conformal normal flow. This is the idea of Brendle for closed hypersurfaces in a warped product in \cite{Br13}. In our case, we need to
pay extra attention on the boundary.  We first prove that $\Phi$ is surjective.
\begin{lemma} $\O\subset \Phi(A)$.
\end{lemma}
\begin{proof} For each $q\in \O$, let $p\in \S$ be the point such that $t_0:=d_g(p, q)=d_g(\S, q)$.
We shall prove that $p\in {\rm int}(\S)$. If this is the case, then it is clear that $\Phi(p, t_0)=q$.

Suppose by contradiction that $p\in \p \S$. Let $B_g(q, t_0)$ be the $g$-geodesic ball of radius $t_0$ centered at $q$. Denote $\nu_\delta^B$ and $\nu_g^B$ be the outward $\delta$-unit normal and $g$-unit normal to  $\p B_g(q, t_0)$ respectively. It is clear that $\nu_g^B=x_{n+1}\nu_\delta^B$. Since $d_g(p, q)=d_g(\S, q)$, we find that
 $$g(\nu_g^B(p), \tau)=0, \quad \forall\tau\in T_p(\p \S).$$ which means $\nu_g^B(p)\in N^\delta_p(\p\S):=\{v\in T_p \rr^{n+1}: \<v, w\>=0 \hbox{ for }w\in T_p(\p \S)\}$. It is clear that $\nu_\delta(p)$ is also in $ N^\delta_p(\p\S)$ and $\bar N(p)\in N^\delta_p(\p\S)$. Hence it follows again from $d_g(p, q)=d_g(\S, q)$ that
$$\<\nu_\delta, \bar N\>(p)\ge \<\nu_\delta^B, \bar N\>(p).$$
Recall that in the hyperbolic Poincare model $(\rr^{n+1}_+, g)$, the $g$-geodesic is a semi-circle (that is orthogonal to $\p \rr^{n+1}_+$). Hence $-\nu_g^B$, which is the tangential vector of the $g$-geodesic from $p$ to $q$, points inwards $\O$, since $q\in \O\subset \mathbb{B}^{n+1}$. It follows that $-\nu_\delta^B$ points inwards $\O$, which implies
$$\<-\nu_\delta^B, \bar N\>(p)<0.$$
However, $\<\nu_\delta, \bar N\>(p)=-\cos\theta(x)\le 0$ by our assumption. This leads to a contradiction.  
\end{proof}

\begin{lemma}\label{xlem}
 If the mean curvature of $\S$ is positive, Then the mean curvature $H(p, t)$ of $\S_t^*$ remains positive and satisfies
\begin{eqnarray}\label{mean-curv-evol}
\frac{\p}{\p t}\frac{x_{n+1}}{H}\le -\frac{1}{n}x_{n+1}^2.
\end{eqnarray}
\end{lemma}
\begin{proof}
The mean curvature $H$ satisfies the evolution equation:
\begin{eqnarray*}
\frac{\p}{\p t}H=\Delta_{\S_t^*}x_{n+1}+|h|^2x_{n+1}\ge -H\<\nu_\delta, E_{n+1}\>+\frac{H^2}{n}x_{n+1},
\end{eqnarray*}
where $h$ denotes the second fundamental form of  $\S_t^*$.
One can also check easily that $x_{n+1}$ satisfies the evolution equation:
\begin{eqnarray*}
\frac{\p}{\p t}x_{n+1}=-x_{n+1}\<\nu_\delta, E_{n+1}\>.
\end{eqnarray*}
It follows that
\begin{eqnarray*}
\frac{\p}{\p t}\frac{H}{x_{n+1}}\ge \frac{H^2}{n}.
\end{eqnarray*}
Hence $H$ remains positive and \eqref{mean-curv-evol} follows.
\end{proof}
\begin{lemma}
\begin{eqnarray*}\label{mean-curv-evol1}
\frac{d}{d t}\int_{\S_t^*}\frac{x_{n+1}}{H}\le -\int_{\S_t^*}\frac{n+1}{n}x_{n+1}^2.
\end{eqnarray*}
\end{lemma}
\begin{proof} From \eqref{mean-curv-evol}, we have
\begin{eqnarray*}
\frac{d}{d t}\int_{\S_t^*}\frac{x_{n+1}}{H}\le \int_{\S_t^*} \left(-\frac{1}{n}x_{n+1}^2\right)+\frac{x_{n+1}}{H} (-Hx_{n+1})=- \int_{\S_t^*}\frac{n+1}{n}x_{n+1}^2.\end{eqnarray*}
\end{proof}

Now we can prove Theorem A.

\

\noindent
{\it Proof of Theorem A.}
Let $$Q(t)=\int_{\S_t^*} \frac{x_{n+1}}{H}.$$
It follows that for $\tau\in (0, T)$,
\begin{eqnarray*}
Q(0)-Q(\tau)&= &-\int_0^\tau \left( \frac{d}{d t}\int_{\S_t^*}\frac{x_{n+1}}{H}\right) dt\\&\ge& \frac{n+1}{n} \int_0^\tau \int_{\S_t^*} x_{n+1}^2=\frac{n+1}n \int_{\Phi(A^*\cap (\S\times [0, \tau]))} x_{n+1}.
\end{eqnarray*}
Letting $\tau\to T$ and recalling that $\O\subset \Phi(A)$ and $A\setminus A^*$ has measure zero, we conclude
\begin{eqnarray*}
Q(0)=\int_{\S} \frac{x_{n+1}}{H}\ge\frac{n+1}{n}  \int_{\O} x_{n+1}.
\end{eqnarray*}
If equality holds, then by the proof of Lemma \ref{xlem}, we see that for all $t\in[0, T)$, $\S_t^*$ is umbilical. In particular, $\S=\S_0^*$ is umbilical, which is a spherical cap. It is easy to check that equality holds only if it is a free boundary spherical cap.\qed

\begin{remark}
In the proof given above we have  used the geodesics with respect to the hyperbolic metric $g$, a Riemannian metric.     For the capillary hypersurfaces studied in the next section, we will use geodesics with respect to a Finsler metric $F$ given in \eqref{finsler}. It indicates that the difference between the geometry of free boundary hypersurfaces and the one of capillary hypersurfaces could be big. One belongs to Riemannian geometry, while the other Finslerian geometry.

\end{remark}

\

\section{Zermelo's navigation and Finsler metrics}
We shall give a brief introduction of Zermelo's navigation problem and the induced Finsler metric of Randers-type, as well as some properties we shall use. We refer to \cite{BRS04} for a detailed description on Zermelo's navigation problem. For Finsler geometry and Randers metrics see for example \cite{BCS00} and \cite{CS12}.

Let $\mathbb{R}^{n+1}$ be equipped with a Riemannian metric $g$ and ${\bf v_0}$  a Euclidean constant vector field on $\mathbb{R}^{n+1}$ such that $|{\bf v_0}|_g<1$. 
Zermelo's navigation problem asks to find a Finsler metric $F$ on $\mathbb{R}^{n+1}$ such that for any non-zero vector  $\xi\in T_x\mathbb{R}^{n+1}$,  \begin{eqnarray}\label{navi1}
\left|\frac{\xi}{F(x, \xi)}+{\bf v_0}\right|_{g_x}=1.
\end{eqnarray}
We call $F$ the navigation Finsler metric with respect to $(g, {\bf v_0})$.
It turns out that $F$ is given by
\begin{eqnarray}\label{finsler}
F(x, \xi)=\sqrt{\frac{|\xi|_{g_x}^2}{1-|{\bf v_0}|_{g_x}^2}+\frac{g_x(\xi, {\bf v_0})^2}{(1-|{\bf v_0}|_{g_x}^2)^2}}+\frac{g_x(\xi, {\bf v_0})}{1-|{\bf v_0}|_{g_x}^2}.
\end{eqnarray}
Indeed, \eqref{navi1} is equivalent to $$(1-|{\bf v_0}|_{g_x}^2)F^2(x,\xi)-2{g_x}(\xi, {\bf v_0})F(x, \xi)-|\xi|_{g_x}^2=0.$$ Then \eqref{finsler} follows.

$F$ is in fact a Randers metric, which is a very special  Finsler metric of the form $$F(x,\xi)=\sqrt{\alpha_x(\xi,\xi)}+\beta_x(\xi),$$
where 
\begin{equation}
\label{alpha}\alpha_x=\frac{g_x(\cdot, \cdot)}{1-|{\bf v_0}|_{g_x}^2}+\frac{g_x(\cdot, {\bf v_0})^2}{(1-|{\bf v_0}|_{g_x}^2)^2}\end{equation} gives a Riemannian metric and 
$$\beta_x(\cdot)=\frac{g_x(\cdot, {\bf v_0})}{1-|{\bf v_0}|_{g_x}^2}$$ gives a $1$-form. Denote
\begin{eqnarray}
g_{\xi}^F(X, Y)=\frac{\p^2}{\p \xi_i\p\xi_j}(\frac12F^2)(x, \xi)X^iY^j,\quad \hbox{ for }X, Y\in T_x \mathbb{R}^{
n+1}.
\end{eqnarray}

We have the following properties for such $F$.
\begin{lemma}\label{finsler-prop1}
For $V\in T_x\mathbb{R}^{n+1}$, we have $$F(x, V-{\bf v_0})=1 \Leftrightarrow |V|_g=1.$$
\end{lemma}
\begin{proof}If $V$ satisfies $F(x, V-{\bf v_0})=1$, then from \eqref{navi1} we see $$ |V|_g=\left|\frac{V-{\bf v_0}}{F(x, V-{\bf v_0})}+{\bf v_0}\right|_g=1.$$
Conversely, if $ |V|_g=1$, we let $\zeta=V-{\bf v_0}$. Then
\begin{eqnarray}\label{navi2}
1=|\zeta+{\bf v_0}|_g^2=|\zeta|_g^2+2g(\zeta, {\bf v_0})+|{\bf v_0}|_g^2.
\end{eqnarray}
On the other hand, by the definition \eqref{finsler} of $F$, we check that for any $\xi$,
\begin{eqnarray*}
F(x,\xi)=1&&\Leftrightarrow \frac{|\xi|_{g_x}^2}{1-|{\bf v_0}|_{g_x}^2}+\frac{g_x(\xi, {\bf v_0})^2}{(1-|{\bf v_0}|_{g_x}^2)^2}=\left(1-\frac{g_x(\xi, {\bf v_0})}{1-|{\bf v_0}|_{g_x}^2}\right)^2 \\&&\Leftrightarrow  \frac{|\xi|_{g_x}^2}{1-|{\bf v_0}|_{g_x}^2}=1- 2\frac{g_x(\xi, {\bf v_0})}{1-|{\bf v_0}|_{g_x}^2}.
\end{eqnarray*}
Hence it follows from \eqref{navi2} that $F(x,\zeta)=1$.
\end{proof}

\begin{lemma}\label{finsler-prop2'}Let $\zeta=V- {\bf v_0}$ for some vector $V\in T_x\mathbb{R}^{n+1}$ with $|V|_g=1$. Then for any $W\in T_x\mathbb{R}^{n+1}$,
\begin{eqnarray}\label{finsler1}
&&g^F_{\zeta}(\zeta, W)= \frac{1}{1-g(\zeta+{\bf v_0}, {\bf v_0})} g(\zeta+{\bf v_0}, W).
\end{eqnarray}
\end{lemma}
\begin{proof}
By a direct computation, we see that
\begin{eqnarray}\label{finsler2'}
&&D_{\xi_i} F(x, \zeta)=\frac{1}{\sqrt{\frac{|\xi|_{g_x}^2}{1-|{\bf v_0}|_{g_x}^2}+\frac{g_x(\xi, {\bf v_0})^2}{(1-|{\bf v_0}|_{g_x}^2)^2}}}\frac{1}{1-|{\bf v_0}|_{g_x}^2}g_{ij}(\zeta^j+{\bf v_0}^jF(x, \zeta)).
\end{eqnarray}
By the $1$-homogeneity of $F$, we know $\sum\limits_j\frac12\frac{\p^2 F}{\p \xi_i\p\xi_j}(x, \zeta)\zeta_j=0$. Thus 
\begin{eqnarray*}
g_{\zeta}^F(\zeta, W)&=F(x, \zeta) D_{\xi_i} F(x, \zeta) W^i
\end{eqnarray*}
Using the definition of $F$, \cref{finsler-prop1} and \eqref{finsler2'}, we get the assertion.
\end{proof}

For each $x$, denote by $F_x=F(x, \cdot)$ the Minkowski gauge and $F_x^*: (\rr^{n+1})^*\to \rr_+$  the dual  Minkowski gauge  to $F_x$. Let $l_x: \rr^{n+1}\to (\rr^{n+1})^*$ be the Legendre transform given by
$$l_x(\xi)(W)=g_{\xi}^{F_x}(\xi, W),$$
which is equal to $D\frac 12 (F_x)^2,$
and $l^*: (\rr^{n+1})^*\to \rr^{n+1}$ be its dual, which is defined by
$$l_x^*=D\frac12(F_x^*)^2.$$ 
It is well-known that
\begin{eqnarray}\label{dual-mink}
F_x(\xi)= F_x^*(l(\xi)),\quad l_x^*=(l_x)^{-1}.
\end{eqnarray}

As a direct consequence of \cref{finsler-prop2'}, we get
\begin{lemma}\label{finsler-coro2} Let $\zeta=V- {\bf v_0}$ for some vector $V\in T_x\mathbb{R}^{n+1}$ with $|V|_g=1$. Then
\begin{eqnarray}
\label{finsler3}
&&l(\zeta)= \frac{1}{1-g(\zeta+{\bf v_0}, {\bf v_0})} (\zeta+{\bf v_0})^{\sharp_g},
\end{eqnarray}
where the superscript $^{\sharp_g}$ denotes the dual $1$-form with respect to $g$.
In particular, for $\zeta$ such that $F(x, \zeta)=1$ and $W\in T_x\mathbb{R}^{n+1}$
we have
\begin{eqnarray}\label{finsler4}
&&l(\zeta)(W)=0 \Leftrightarrow g(\zeta+{\bf v_0}, W)=0. 
\end{eqnarray}
\begin{eqnarray}\label{finsler5}
&&DF^*((\zeta+{\bf v_0})^{\sharp_g})= \zeta,
\end{eqnarray}
\end{lemma}
\begin{proof}We need only to prove \eqref{finsler5}.
Since $F(x,\zeta)=1$, we see from \eqref{dual-mink} that $$F^*(x, l(\zeta))=F\left(x, \frac{1}{1-g(\zeta+{\bf v_0}, {\bf v_0})}(\zeta+{\bf v_0})^{\sharp_g}\right)=1.$$ 
From  \eqref{dual-mink} and \eqref{finsler3}, we have
$$l^*_x\left(\frac{1}{1-g(\zeta+{\bf v_0}, {\bf v_0})}(\zeta+{\bf v_0})^{\sharp_g}\right)=(l_x)^{-1}\left(\frac{1}{1-g(\zeta+{\bf v_0}, {\bf v_0})}(\zeta+{\bf v_0})^{\sharp_g}\right)= \zeta.$$
It follows that $$DF^*((\zeta+{\bf v_0})^{\sharp_g})=DF^*\left(\frac{1}{1-g(\zeta+{\bf v_0}, {\bf v_0})}(\zeta+{\bf v_0})^{\sharp_g}\right)=\zeta.$$
\end{proof}

\

\section{Capillary hypersurfaces and Heintze-Karcher-type inequality}

\subsection{Zermelo's navigation and capillary hypersurfaces in a ball}\

Let $g={x_{n+1}^{-2}}\delta$ be the hyperbolic metric on $\rr^{n+1}_+$ and ${\bf v_0}=\cos \theta_0 E_{n+1}$.
To ensure $|{\bf v_0}|_g<1$, we assume that $x_{n+1}>|\cos \theta_0|$. That is, we work in $$\mathbb{R}^{n+1}_{\theta_0}:= \{x_{n+1}>|\cos \theta_0|\}.$$
Let $F$ be the navigation Finsler metric on $\mathbb{R}^{n+1}_{\theta_0}$ with respect to $(g, {\bf v_0})$.
We have the following simple by important observation.
\begin{proposition}
$\beta_x$ is a closed $1$-form.
\end{proposition}
\begin{proof}
Note that $$\beta_x=\beta_i(x)dx^i = \frac{x_{n+1}^{-2}}{1-x_{n+1}^{-2}\cos^2\theta}\<{\bf v_0}, E_i\>dx^i=\frac{x_{n+1}^{-2}}{1-x_{n+1}^{-2}\cos^2\theta}\cos \theta dx^{n+1}
.$$
It is direct to check that $$\frac{\p}{\p x_j}{\beta_i}=\frac{\p}{\p x_i}{\beta_j}, \quad i, j=1,\cdots, n+1.$$
The assertion follows.
\end{proof}

The following is a well-known fact in Finsler geometry, see for example the book of Bao-Chern-Shen \cite{BCS00}.
\begin{proposition}[\cite{BCS00} Exercise 11.3.4]\label{geod-same}
Let $F(x,\xi)=\sqrt{\alpha_x(\xi,\xi)}+\beta_x(\xi)$ be a Randers metric with  a Riemannian metric $\alpha$ and a one-form $\beta$. If $\beta$ is closed,
then the $F$-geodesic i.e., the geodesic with respect to the Finsler metric $F$, has the same trajectories as the $\alpha$-geodesic.
\end{proposition}

Let $\S\subset \overline{\mathbb{B}^{n+1}}\cap \rr^{n+1}_{\theta_0}$ be a compact hypersurface with boundary $\p\S\subset \p \mathbb{B}^{n+1}$ such that $\p\S$ intersects $\p \mathbb{B}^{n+1}$ transversally with  contact angle  $\theta(x)\le \theta_0$ for each $x\in \p\S$. Assume that $\S$ is strictly mean convex. Let $\O$ be the domain enclosed by $\S$ and $\mathbb{S}^n_+$. 

Let ${\bf n}_\delta, {\bf n}_g, {\bf n}_F$ be the  $\delta$-, $g$-, $F$-unit normal to $\S$ pointing inward $\O$ respectively. In particular, 
$$|{\bf n}_g|_g=1, \quad g({\bf n}_g, W)=0\hbox{ for }W\in T\S.$$ Denote $$\zeta={\bf n}_g-{\bf v_0}.$$
 From \cref{finsler-prop1} and \cref{finsler-coro2}, we see $$F(\cdot, \zeta)=1, \quad g^F_\zeta(\zeta, W)=0 \hbox{ for }W\in T\S.$$
By the definition of ${\bf n}_F$, we know that
$$F(\cdot, {\bf n}_F)=1,\quad g^F_{{\bf n}_F}({\bf n}_F, W)=0,\quad W\in T\S.$$ It follows that \begin{eqnarray}\label{normal-rel}
{\bf n}_F=\zeta={\bf n}_g-{\bf v_0}=x_{n+1}{\bf n}_\delta-{\bf v_0}. 
\end{eqnarray}
\begin{remark} If $\S$ is a capillary hypersurface in  $\bb^{n+1}$ which intersects $\p \bb^{n+1}$ at the constant contact angle $\theta_0$, one sees $$\<x_{n+1}{\bf n}_\delta-{\bf v_0}, \bar N(x)\>=\<x_{n+1}{\bf n}_\delta-{\bf v_0}, x\>=0,$$which is equivalent to $$g^F_{{\bf n}_F}({\bf n}_F, \bar N(x))=0.$$This means that, by using the navigation Finsler metric, the capillary boundary condition for the Euclidean metric is transferred to a free boundary condition  for the navigation Finsler metric. 
\end{remark}
For $p\in \S$, let $\exp_p^F$  the exponential map at $p$ with respect to $F$, and $\Phi$ be the $F$-normal exponential map from $\S$ given by\begin{eqnarray*}
\Phi: &&\S\times [0,\infty)\to \rr^{n+1}_+,\\
&&\Phi(p, t)=\exp_p^F(t{\bf n}_F(p)).
\end{eqnarray*}
Let $d_F(\S, q)$ be the $F$-distance function  from $\S$ to $q$ and $c: \S\to [0,\infty)$  the $F$-cut function of $\S$, that is, $$c(p)=\sup\{t\in [0, \infty): d_F(\S, \Phi(p, t))=t\}.$$ Denote $T=\max\limits_{p\in \S} c(p)$.
We define
\begin{eqnarray*}
&A=\{(p, t)\in \S\times [0, T): d_F(\S, \Phi(p, t))=t\},\\
&A^*=\{(p, t)\in \S\times [0, T): (p, t+\delta)\in A\hbox{ for some }\delta>0\}.
\end{eqnarray*}
As in the  Riemannian case, it is known  that $A^*$ is an open set of $\S\times [0, T)$ and the restriction $\Phi|_{A^*}$ is a diffeomorphism.
    $\Phi(A)$ is a closed set and $\Phi(A\setminus A^*)$ is the $F$-cut locus ${\rm Cut}_F(\S)$, which is of ${\mathcal H}^{n+1}$-measure zero, see for example \cite{LN05}.
Next we prove $\Phi$ is surjective. Compare to the case $\theta_0=\frac{\pi}{2}$, a key lemma for anisotropic gauge in \cite{JWXZ23} will be essentially used.
\begin{lemma} \label{lem4.4} $\O\subset \Phi(A)$.
\end{lemma}
\begin{proof} For each $q\in \O$, let $p\in \S$ be the point such that $t_0:=d_F(p, q)=d_F(\S, q)$.
We shall prove that $p\in {\rm int}(\S)$. If this is the case, then it is clear that $\Phi(p, t_0)=q$.

Suppose there exists some minimal point $p\in \p \S$. Let $B_F(q, t_0)$ be the backward $F$-geodesic ball of radius $t_0$ centered at $q$, that is, $$B_F(q, t_0)=\{y\in \mathbb{R}^{n+1}_\theta: d_F(y, q)<t_0\}.$$ Thus $\S\subset\mathbb{R}^{n+1}_\theta\setminus B_F(q, t_0)$. Denote ${\bf n}_\delta^B, {\bf n}_g^B, {\bf n}_F^B$ be the inward $\delta$-, $g$-, $F$-unit normal to $B_F(q, t_0)$ respectively. Similarly as above,  $${\bf n}_F^B= {\bf n}^B_g-{\bf v_0}=x_{n+1}{\bf n}_\delta^B-{\bf v_0}.$$
Since $d_F(p, q)=d_F(\S, q)$, we find that
 $$g^F_{{\bf n}_F^B(p)}({\bf n}_F^B(p), \tau)=0, \quad \forall\tau\in T_p(\p \S).$$ In view of  \cref{finsler-coro2}, we see that $${\bf n}_\delta^B(p) \parallel {\bf n}_F^B(p)+{\bf v_0}\in N^\delta_p(\p\S):=\{v\in T_p \rr^{n+1}: \<v, w\>=0 \hbox{ for }w\in T_p(\p \S)\}.$$It is clear that ${\bf n}_\delta(p)\in N_p^\delta(\p\S)$ and $\bar N(p)\in N_p^\delta(\p\S)$. Hence it follows again from $d_F(p, q)=d_F(\S, q)$ that
$$\<{\bf n}_\delta, \bar N\>(p)\le \<{\bf n}_\delta^B, \bar N\>(p).$$
Using the key lemma  given in \cite[Proposition 3.1]{JWXZ23}, we get
$$\<DF^*(p, {\bf n}_\delta), \bar N\>(p)\le \<DF^*(p, {\bf n}_\delta^B), \bar N\>(p).$$
In view of \eqref{finsler5}, we see that $$DF^*(p, {\bf n}_\delta)=DF^*(p, {\bf n}_F+{\bf v_0})={\bf n}_F,\quad DF^*(p, {\bf n}_\delta^B)=DF^*(p, {\bf n}_F^B+{\bf v_0})={\bf n}_F^B.$$
Thus
$$\<{\bf n}_F, \bar N\>(p)\le \<{\bf n}_F^B, \bar N\>(p).$$
Next, we claim that \begin{eqnarray}\label{claim}
\<{\bf n}_F^B, \bar N\>(p)<0.
\end{eqnarray}
To achieve this, we need to see how the $F$-geodesic behaves.
Indeed, by \cref{geod-same}, we see the trajectory of $F$-geodesic is the same as that of the $\alpha$-geodesic, where $\alpha$-geodesic is the geodesic with respect to the Riemanian metric part $\alpha$ of the Randers metric $F$. Recall that $\alpha$ is given  by \eqref{alpha}.
In the appendix we show that $\mathbb{B}_+^{n+1}$ is geodesically strictly convex in $(\rr^{n+1}_+, \alpha)$ (see \cref{geod-convex}). Thus, since $q\in \O\subset \mathbb{B}_+^{n+1}$, we see that $\alpha$-geodesic from $p$ to $q$ is contained in $\O$. As a consequence, the $F$-geodesic from $p$ to $q$ is contained in $\O$.
This implies that ${\bf n}_F^B$, which is exactly the $F$-unit tangential vector of $F$-geodesic from $p$ to $q$, points inward $\O$. This yields the claim \eqref{claim}.
However, the claim contradicts the assumption $\theta(p) \le \theta_0$, because it implies
together with \eqref{normal-rel}  $$\<{\bf n}_F, \bar N\>(p)=\<x_{n+1}{\bf n}_\delta-{\bf v_0}, \bar N\>(p)=x_{n+1} (\cos \theta(x)-\cos \theta_0)\ge 0.$$ This leads to a contradiction.
\end{proof}

\medskip

\subsection{Heintze-Karcher-type inequality in a ball}\

New we prove Theorem \ref{thm1.2} in this subsection.

For each $t\in [0, T)$, define $$\S_t^*=\Phi(A^*\cap (\S\times\{t\})).$$
Then $\S_t^*$ is a smooth hypersurface contained in $\{d_F(\S, \cdot)=t\}$. 
It is clear that $\S_t^*$ solves the flow equation 
\begin{eqnarray}\label{geod-normal-flow0}
\p_t \Phi(p, t) &=&{\bf n}_F(p, t)=x_{n+1}{\bf n}_\delta(p, t)-{\bf v_0}\nonumber\\
&=&\left(x_{n+1}- \<{\bf n}_\delta, {\bf v_0}\>\right){\bf n}_\delta(p, t)- {\bf v_0}^T,
\end{eqnarray}
where ${\bf v_0}^T\in T\S_t^*$ is the tangential component of ${\bf v_0}$.
In the following, we will only use ${\bf n}_\delta$, which we abbreviate as ${\bf n}$.

We need the following evolution equations for corresponding geometric quantities. For the proof, we refer to \cite[Proposition 2.11]{WeX22}.\begin{proposition}\label{basic evolution eq}
Let $x(\cdot, t)$ be a hypersurface flow satisfying \begin{eqnarray*}\label{flow with normal and tangential}
	\p_t x=f{\bf n}+\tau,
\end{eqnarray*} where $\tau$ is tangential part. We have
\begin{enumerate} 
	\item $\p_tdA_t =\left(-fH+\div (\tau)\right)dA_t.$ 
	\item $\p_t {\bf n} =-\n f-h(e_i,T)e_i$.	
	\item $\p_t H=\Delta f+|h|^2 f+ \n_\tau H$.
\end{enumerate}
\end{proposition}  
For \eqref{geod-normal-flow0}, we have the following evolution equation.
\begin{lemma} If the mean curvature of $\S$ is positive, Then the mean curvature $H(p, t)$ of $\S_t^*$ remains positive and it satisfies
\begin{eqnarray}\label{mean-curv-evol'}
\frac{\p}{\p t}\frac{x_{n+1}-\<{\bf n}, {\bf v_0}\>}{H}\le -\frac{1}{n}x_{n+1}(x_{n+1}-\<{\bf n}, {\bf v_0}\>).
\end{eqnarray}
\end{lemma}
\begin{proof}
The mean curvature $H$ satisfies the evolution equation:
\begin{eqnarray*}
\frac{\p}{\p t}H&=&\Delta_{\S_t^*}(x_{n+1}- \<{\bf n}, {\bf v_0}\>)+|h|^2(x_{n+1}- \<{\bf n}, {\bf v_0}\>)-\nabla_{{\bf v_0}^T} H
\\&=&H\<{\bf n}, E_{n+1}\>+\< |h^2|{\bf n}+\nabla H, {\bf v_0}\>+|h|^2(x_{n+1}- \<{\bf n}, {\bf v_0}\>)-\nabla_{{\bf v_0}^T} H
\\&= &H\<{\bf n}, E_{n+1}\>+|h|^2x_{n+1}.
\end{eqnarray*}
One can also check that  $x_{n+1}-\<{\bf n}, {\bf v_0}\>$ satisfies the evolution equation:
\begin{eqnarray*}
\frac{\p}{\p t}(x_{n+1}-\<{\bf n}, {\bf v_0}\>)&=&\<{x_{n+1}\bf n}-{\bf v_0}, E_{n+1}\>+\<\nabla (x_{n+1}- \<{\bf n}_\delta, {\bf v_0}\>), {\bf v_0}\>-h({\bf v_0}^T, {\bf v_0}^T)
\\&=&\left(x_{n+1}- \<{\bf n}, {\bf v_0}\>\right)\<{\bf n}, E_{n+1}\>.
\end{eqnarray*}
It follows that
\begin{eqnarray*}
\frac{\p}{\p t}\frac{H}{x_{n+1}-\<{\bf n}, {\bf v_0}\>}\ge \frac{H^2}{n}\frac{x_{n+1}}{x_{n+1}-\<{\bf n}, {\bf v_0}\>}.
\end{eqnarray*}
Hence $H$ remains positive and \eqref{mean-curv-evol'} follows.
\end{proof}

\begin{lemma}
\begin{eqnarray*}
\frac{d}{d t}\int_{\S_t^*}\frac{x_{n+1}-\<{\bf n}, {\bf v_0}\>}{H}\le -\frac{n+1}{n}\int_{\S_t^*}x_{n+1}(x_{n+1}-\<{\bf n}, {\bf v_0}\>).
\end{eqnarray*}
\end{lemma}
\begin{proof} From \eqref{mean-curv-evol}, we have
\begin{eqnarray*}
&&\frac{d}{d t}\int_{\S_t^*}\frac{x_{n+1}-\<{\bf n}, {\bf v_0}\>}{H}\\&\le &\int_{\S_t^*} \left(-\frac{1}{n}x_{n+1}(x_{n+1}-\<{\bf n}, {\bf v_0}\>)\right)+
\frac{x_{n+1}-\<{\bf n}, {\bf v_0}\>}{H} \left[-H  (x_{n+1}-\<{\bf n}, {\bf v_0}\>)-{\rm div}_{\S_t^*}{\bf v_0}^T\right]
\\&=&\int_{\S_t^*} \left(-\frac{1}{n}x_{n+1}(x_{n+1}-\<{\bf n}, {\bf v_0}\>)\right)+
\frac{x_{n+1}-\<{\bf n}, {\bf v_0}\>}{H} \left[-H  (x_{n+1}-\<{\bf n}, {\bf v_0}\>)-H\<{\bf n}, {\bf v_0}\>\right]
\\&=&
-\frac{n+1}{n} \int_{\S_t^*}x_{n+1}(x_{n+1}-\<{\bf n}, {\bf v_0}\>).\end{eqnarray*}
\end{proof}

\noindent{\it Proof of Theorem \ref{thm1.2}.}
Let $$Q(t)=\int_{\S_t^*} \frac{x_{n+1}-\<{\bf n}, {\bf v_0}\>}{H}.$$
It follows that for $\tau\in (0, T)$,
\begin{eqnarray*}Q(0)-Q(\tau)&=& -\int_0^\tau \left( \frac{d}{d t}\int_{\S_t^*}\frac{x_{n+1}-\<{\bf n}, {\bf v_0}\>}{H}\right) dt\\&\ge &\frac{n+1}{n} \int_0^\tau \int_{\S_t^*} x_{n+1}(x_{n+1}-\<{\bf n}, {\bf v_0}\>)\\&=&\frac {n+1} n\int_{\Phi(A^*\cap (\S\times [0, \tau]))} x_{n+1}.
\end{eqnarray*}
Let $\tau\to T$, recalling that $\O\subset \Phi(A)$ and $A\setminus A^*$ has measure zero, we conclude
\begin{eqnarray*}
Q(0)=\int_{\S} \frac{x_{n+1}-\<{\bf n}, {\bf v_0}\>}{H}\ge\frac{n+1}{n}  \int_{\O} x_{n+1}.
\end{eqnarray*}
The equality characterization for this case follows similarly as in free boundary case.

\medskip

\subsection{Heintze-Karcher-type inequality in the half-space}\label{4.3}\ 

In this subsection, we remark that the approach to Theorem \ref{thm1.2} can also be easily adapted to prove 
 the following Heintze-Karcher inequality for capillary hypersurfaces in the half-space $\mbR^{n+1}_+$, which has been proved in our previous work \cite{JWXZ23} together with Jia and Zhang. 

\medskip

\noindent{\bf Theorem B} (\cite{JWXZ22})
{\it   Let $\theta_0\in(0, \pi)$ and let $\S\subset\overline{\rr^{n+1}_+}$ be a smooth, compact, embedded, strictly mean convex $\theta$-capillary hypersurface, with  $\theta(x)\le \theta_0$ for every $x\in \partial\Sigma$. Let $\O$ denote the enclosed domain by $\S$ and $\p \rr^{n+1}_+$. Then it holds
    \begin{eqnarray}\label{EQ-HK-halfspace}
 \int_\S \frac{1-\cos\theta_0\<\nu, E_{n+1}\>}{H}  \geq\frac{n+1}{n}|\O|.
    \end{eqnarray}
    Equality in \eqref{EQ-HK-halfspace} holds if and only if $\S$ is a $\theta_0$-capillary spherical cap.}

\medskip

For Theorem B,  the Finsler metric $F$ of Randers-type in \eqref{finsler} we shall choose is generated by navigation data $(\delta, -\cos\theta_0 E_{n+1})$. This is indeed a Minkowski metric, that is, $F$ in \eqref{finsler} generated by navigation data $(\delta, -\cos\theta_0 E_{n+1})$ does not depend on $x$. Thanks to it, the geodesic normal flow is much simpler. Precisely, the geodesics for $F$ are given by straight lines and the geodesic normal flow from $\S$ is given by \begin{eqnarray}\label{geod-normal-flow}
x(p, t)=x(p)-t(\nu(p)-\cos\theta_0 E_{n+1}), \quad p\in \S.
\end{eqnarray}
\eqref{geod-normal-flow} is exactly the parallel hypersurfaces defined in \cite{JWXZ22}.
One can verify that $$\int_\S \frac{1-\cos\theta_0\<\nu, E_{n+1}\>}{H}- \frac{n+1}{n}|\O|$$ has a  monotone property under the geodesic normal flow \eqref{geod-normal-flow}. The Heintze-Karcher inequality \eqref{EQ-HK-halfspace} follows. We leave the details to the reader.

\

\appendix

\section{}
In this Appendix we  show that $\{x\in \mathbb{S}^{n}_+ | x_{n+1} > | \cos \theta_0|\}
\subset (\rr^{n+1}_+, \alpha)$ is globally strictly convex, which was used in the proof of Lemma \ref{lem4.4}.

Consider 
$(\{x\in \rr^{n+1}_+ | x_{n+1} > |\cos \theta_0|\},  \alpha)$ 
equipped with Riemannian metric
\begin{eqnarray*}
   \a&=& \frac 1 {x^2_{n+1} -\cos ^2 \theta_0} \sum_{i=1}^n dx_i^2 +\frac{x_{n+1}^{2}}{(x_{n+1}^2-\cos \theta_0)^2} dx_{n+1}^2\\&=:& A\sum_{i=1}^n dx_i^2 +B dx_{n+1}^2.
\end{eqnarray*}

\begin{proposition} $(\{x\in \rr^{n+1}_+ | x_{n+1} > |\cos \theta_0|\},  \alpha)$ has negative sectional curvature.
\end{proposition}

\begin{proof} Since all functions considered  depend only on $x_{n+1}$, for simplicity of notation we use $\p $ to denote $\p_{x_{n+1}}$.
Let $\ell =A^{-1}=x_{n+1}^2-\cos^2\theta_0$. It is clear that
\[
\p A =-2x_{n+1} \ell^{-2}, \quad \p B=2x_{n+1}\ell^{-2} -4x^3_{n+1} \ell^{-3}.
\]
It is easy to check that the non-vanishing Christoffel symbols of the metric $\a$ are
\begin{eqnarray*}
    \Gamma^{n+1}_{ii}=-\frac 12 B^{-1}\p A,\quad
    \Gamma^{i}_{i(n+1)}=\frac 12 A^{-1}\p A,\quad
    \Gamma^{n+1}_{(n+1)(n+1)}&=&\frac 12 B^{-1}\p B,\
\end{eqnarray*}
where $i\le n.$ One can check easily that
for  fixed $i,j\le n$
\[
R^i_{ijj}=\Gamma^i_{i(n+1)} \Gamma ^{n+1}_{jj} =-\frac 14 A^{-1}B^{-1} (\p A)^2=-\ell^{-1},
\]Hence the sectional curvatures of the planes spanned by $\frac{\p}{\p x_i}$ and $\frac{\p}{\p x_j}$ are $-1$, for $i,j\le n$.
Moreover, for fixed $j\le n$
\begin{eqnarray*}
    R^{n+1}_{(n+1) {jj}}
    &=& \p \Gamma^{n+1}_{jj}+\Gamma^{n+1}_{(n+1)q}\Gamma^q_{jj}-\Gamma^{n+1}_{jq}\Gamma^q_{(n+1)j}\\&=&
    -\frac 12 \p (B^{-1} \p A)-\frac{1}{4} B^{-2} \p B\p A+\frac 14 A^{-1}B^{-1} (\p A)^2\\
    &=& -x_{n+1}^{-2}+x_{n+1}^2\ell^{-2} -2 x_{n+1}^4 \ell^{-3}+\ell^{-1}\\
    &=& -x_{n+1}^{-2}-\cos^2\theta_0 x^2_{n+1} \ell ^{-3}-\cos^2\theta_0\ell^{-1}(1+x^2_{n+1}\ell^{-1})
.
\end{eqnarray*}
It follows that the sectional curvature of the planes spanned by $\frac{\p}{\p x_i}$ and $\frac{\p}{\p x_{n+1}}$ is also negative.
\end{proof}

\begin{proposition} $\{x\in \mathbb{S}^{n}_+ | x_{n+1} > |\cos \theta_0|\}\subset (\rr^{n+1}_+, \alpha)$ with the induced metric has positive second fundamental form.
\end{proposition}
 \begin{proof}Since the metric $\alpha$ and $\mathbb{S}^n_+$ is axially symmetric about $x_{n+1}$,  we only need to prove the Proposition for $n=2$. Let $$r(\varphi, \beta)=(x, y, z)=(\sin \varphi\cos\beta, \sin \varphi\sin\beta, \cos \varphi), \quad\varphi\in(0,\theta_0), \,\beta\in (0,2\pi)$$ be the parametrized $\mathbb{S}^2_+$ in $\rr^3_+$.
Then $$r_\varphi=(\cos \varphi\cos\beta, \cos \varphi\sin\beta, -\sin \varphi), \quad r_\beta=(-\sin \varphi\sin\beta, \sin \varphi\cos\beta, 0).$$
\begin{eqnarray*}
&&r_{\varphi\varphi}=(-\sin \varphi\cos\beta, -\sin \varphi\sin\beta, -\cos \varphi),\\
&&r_{\varphi\beta}=(-\cos \varphi\sin\beta, \cos \varphi\cos\beta, 0),\\
&&r_{\beta\beta}=(-\sin \varphi\cos\beta, -\sin \varphi\sin\beta, 0).
\end{eqnarray*}

An inner normal (not unit) to $\mathbb{S}^2_+$ is given by solving $V$ from $\alpha(r)(r_\varphi, V)=0$ and $\alpha(r)(r_\beta, V)=0$.
We may write
$V=(-C\cos\beta, -C\sin\beta, -1)$. It can be solved that
$$C= \frac{\sin\varphi\cos\varphi}{\cos^2\varphi-\cos^2\theta_0}>0.$$
It follows that
\begin{eqnarray*}
\alpha(r_{\varphi\varphi}, V)&=&  AC\sin \varphi+ B\cos\varphi,\\
\alpha(r_{\varphi\beta}, V)&=&0,\\
\alpha(r_{\beta\beta}, V)&=&
AC\sin\varphi.
\end{eqnarray*}
It follows  that 
 $\{x\in \mathbb{S}^{n}_+ | x_{n+1} > | \cos \theta_0|\}$
  in $(\rr^{n+1}_+, \alpha)$ has positive second fundamental form, hence is locally strictly convex.

\end{proof}

Alexander \cite{Al77} proved that a locally strictly convex hypersurface in a negative curved space is globally strictly convex. Hence we get
\begin{corollary}\label{geod-convex}
$\{x\in \mathbb{S}^{n}_+ | x_{n+1} > | \cos \theta_0|\}
\subset (\rr^{n+1}_+, \alpha)$ is globally strictly convex.
\end{corollary}

\

\printbibliography

\end{document}